\documentclass[12pt, reqno]{amsart}
\usepackage{amssymb,a4}
\usepackage{amsmath,amsfonts,amssymb,amsthm}
\usepackage{mathrsfs}
\usepackage{color}
\usepackage{colordvi}

\newcommand{\on}{\operatorname}

\newcommand{\f}{\mathfrak}

\newcommand{\al}{\alpha}

\newcommand{\D}{\mathcal{D}}

\def\wt{{\rm wt}}

\def\de{\delta}

\def\C{{\mathbb C}}

\def\Z{{\mathbb Z}}

\def\N{{\mathbb N}}
\def\1{{\bf 1}}
\def\l{\lambda}

\def \Hom{{\rm Hom}}

\def \<{\langle}
\def \>{\rangle}

\def \wg{{\widehat{\frak{g}}}}
\def \wh{{\widehat{\frak{h}}}}
\def\wo{{\widehat{\frak{so}}}}

\def \sl{\frak{sl}}

\def \l{\lambda}

\def \fg{\frak{g}}

\def \wg{{\widehat{\frak{g}}}}
\def \wh{{\widehat{\frak{h}}}}

\def \sl{\frak{sl}}

\def \l{\lambda}

\numberwithin{equation}{section}
\newtheorem{theorem}{Theorem}[section]

\newtheorem{lem}[theorem]{Lemma}

\newtheorem{remark}[theorem]{Remark}
\theoremstyle{definition}
\newtheorem{defn}[theorem]{Definition}

\begin{document}
\title[Vertex Operator Algebras]{Level-Rank Duality for Vertex Operator Algebras of types $B$ and $D$}

\author{ Cuipo  Jiang} \thanks{Jiang is supported by NSFC Grants11771281,  11531004}
\address[Jiang]{School  of Mathematical Sciences,  Shanghai Jiao Tong University, Shanghai, 200240, China }
\email{cpjiang@sjtu.edu.cn}

\author{Ching Hung Lam}\thanks{Lam is supported by MoST
grant 104-2115-M-001-004-MY3 of Taiwan}
\address[Lam]{Institute of Mathematics, Academia Sinica, Taiwan}
\email{chlam@math.sinica.edu.tw}

\begin{abstract}
For the simple Lie algebra $ \frak{so}_m$, we study the commutant vertex operator
  algebra of $ L_{\widehat{\frak{so}}_{m}}(n,0)$ in the $n$-fold tensor product
  $ L_{\widehat{\frak{so}}_{m}}(1,0)^{\otimes n}$. It turns out that this commutant vertex
  operator algebra can be realized as a fixed point subalgebra of  $L_{\widehat{\frak{so}}_{n}}(m,0)$ (or its simple current extension) associated with a certain abelian group.  This result may be viewed as a version of level-rank duality.
\end{abstract}

\maketitle
\section{Introduction}
Let $\frak{g}$ be a finite dimensional simple Lie algebra and $\widehat{\frak{g}}$ the associated affine Lie algebra. Let $L_{\wg}(1,0)$ be  the basic representation of $\wg$. Then $L_{\wg}(1,0)$ is a rational vertex operator algebra.
  For
  $l \in \N$, the tensor product  $L_\wg(1,0) ^{\otimes l}$ is still rational and the diagonal action of $\hat{\frak{g}}$ on $L_\wg(1,0) ^{\otimes l}$ defines a vertex subalgebra $L_{\wg}(l,0)$ of level $l$.  As a module of the vertex operator algebra $L_{\wg}(l,0)$, $L_\wg(1,0) ^{\otimes l}$  is a  direct sum of irreducible $\wg$-modules:
 \begin{equation} \label{eq:decomposition} L_{\wg}(1,0)^{\otimes l}=\bigoplus L_\wg(l, \bar{\Lambda})\otimes_{\C}M_\wg(l, \bar{\Lambda}),
 \end{equation}
where $L_{\wg}(l, \bar{\Lambda})$ are level $l$ irreducible  $\wg$-modules and $M_\wg(l, \bar{\Lambda})=\Hom_{\wg}(L_\wg(l, \bar{\Lambda}), L_{\wg}(1,0)^{\otimes l})$ are vector spaces. The subspace $M_\wg(l, 0)$ is a vertex operator algebra, which is  called the commutant (or coset) of $L_{\wg}(l,0)$ in
$L_\wg(1,0) ^{\otimes l}$ and is denoted by $C_{L_{\wg}(1,0)^{\otimes l}}(L_{\wg}(l,0))$. Commutant vertex operator algebras  initiated in \cite{GKO1, GKO2} were first introduced in [FZ] from the point view of vertex operator algebras. Since then describing commutant vertex operator algebras has been one of the most interesting questions in the theory of vertex operator algebras. Many interesting examples, especially coset vertex operator algebras related to affine vertex operator algebras,  have been  extensively studied both in the physics and mathematics literatures  \cite{AP}, \cite{ALY1}, \cite{ALY2}, \cite{BEHHH}, \cite{BFH},  \cite{ChL}, \cite{CL}, \cite{DJX}, \cite{DLWY}, \cite{DLY1}, \cite{DLY2}, \cite{DW1}, \cite{DW2}, \cite{GQ}, \cite{JL1}, \cite{JL2}, \cite{La}, \cite{LS}, \cite{LY}, etc.

For $k\in \Z_{+}$, let $L_{\wh}(k,0)$ be the Heisenberg vertex operator subalgebra of $L_{\wg}(k,0)$ associated with a Cartan subalgebra.  The  commutant of $L_{\wh}(k,0)$ in $L_{\wg}(k,0)$, denoted by $ K(\fg, k)$, is the so called  parafermion vertex operator algebra \cite{ZF}. Parafermion vertex operator algebras have been studied extensively \cite{ALY1}-\cite{ALY2}, \cite{CGT}, \cite{DLWY}, \cite{DLY2},   \cite{DW1}-\cite{DW2}, \cite{JL1}, \cite{LY},  etc.
It was proved in \cite{La} and \cite{JL2} independently that $C_{L_{\widehat{\frak{sl}_n}}(1,0)^{\otimes l}}(L_{\widehat{\frak{sl}_n}}(l,0)) \cong
 K(\sl_l,n)$ as vertex operator algebras, which presents a version of level-rank duality. More generally, given a sequence of positive integers
$\underline{\ell}=(l_1, \cdots l_s)$, the tensor product vertex operator algebra
$L_{\wg}(\underline{\ell}, 0)= L_{\wg}(l_1, 0)\otimes L_{\wg}(l_2, 0)\otimes
\cdots \otimes L_{\wg}(l_s,0)$
has a vertex operator subalgebra isomorphic to
$ L_\wg(|\underline{\ell}|,0)$ with $|\underline{\ell}|=l_1+\cdots+l_s$.  The sequence
$\underline{\ell}$ defines a Levi subalgebra $\frak{l}_{\underline{\ell}}$ of
$\frak{sl}_{|\underline{\ell}|}$. Denote by $L_{\widehat{\frak{l}_{\underline{\ell}}}}(n,0)$ the vertex operator subalgebra of
$ L_{\widehat{\frak{sl}_{|\underline{\ell}|}}}(n, 0)$ generated by
$\frak{l}_{\underline{\ell}}$.
Set $K(\frak{sl}_{|\underline{\ell}|}, \frak{l}_{\underline{\ell}}, n)=
C_{L_{\widehat{\frak{sl}_{|\underline{\ell}|}}}(n, 0)}(L_{\widehat{\frak{l}_{\underline{\ell}}}}(n, 0)).$
It was established in \cite{JL2} that $C_{L_{\widehat{\frak{sl}_n}}(\underline{\ell}, 0)}(L_{\widehat{\frak{sl}_n}}(|\underline{\ell}|, 0))\cong K(\frak{sl}_{|\underline{\ell}|}, \frak{l}_{\underline{\ell}}, n)$,  presenting a more general version of level-rank duality.

In this article, we try to generalize the results in \cite{La} and \cite{JL2} to a complex finite-dimensional simple Lie algebra $\f{g}$  of type $B$ or $D$. For $l\in \N$, let $L_\wg(1,0) ^{\otimes l}$ be the tensor product of the affine vertex operator algebra $L_\wg(1,0)$. The main aim is to study the commutant of $L_\wg(l,0)$ in $L_\wg(1,0) ^{\otimes l}$.
 We will determine the structure of this commutant and establish a version of  level-rank duality for these two cases. The idea is to embed the tensor product vertex operator algebra into a large  lattice vertex operator algebra (or superalgebra) associated with a root lattice of type $D$ or $B$ and to use  the fermionic construction of the affine vertex operator algebra $ L_{\widehat{\frak{so}}_{m}}(1,0)$.  It turns out that the tensor product
$ L_{\widehat{\frak{so}}_{m}}(1,0)^{\otimes n}$ is isomorphic to the orbifold VOA $L_{\wo_{mn}}(1,0)^{G}$ for some abelian subgroup $G < Aut(L_{\wo_{mn}}(1,0))$. Using this fact and some results on conformal embeddings of vertex operator algebras given in \cite{KFPX}, we show that

(1) For $m, n\geq 4$,
$$
C_{L_{\wo_{m}}(1,0)^{\otimes n}}(L_{\wo_{m}}(n,0))=L_{\wo_{n}}(m,0)^G$$
if $m$ or $n$ is odd, and
$$
C_{L_{\wo_{m}}(1,0)^{\otimes n}}(L_{\wo_{m}}(n,0))=(L_{\wo_{n}}(m,0)\oplus L_{\wo_{n}}(m,m\Lambda_1))^G$$ if both $m, n$ are even, where  we mean $L_{\widehat{\frak{sl}_2}}(m,m)\otimes L_{\widehat{\frak{sl}_2}}(m,m)$ by $L_{\wo_{n}}(m,m\Lambda_1)$ if $n=4$.

(2) For $m\geq 4$ and $n=3$,
$$C_{L_{\wo_{m}}(1,0)^{\otimes 3}}(L_{\wo_{m}}(3,0))=(L_{\widehat{\frak{sl}_2}}(2m,0)\oplus L_{\widehat{\frak{sl}_2}}(2m,2m))^G$$
if $m$ is even, and
$$C_{L_{\wo_{m}}(1,0)^{\otimes 3}}(L_{\wo_{m}}(3,0))=L_{\widehat{\frak{sl}_2}}(2m,0)^G$$
if $m$ is odd.

 The paper is organized as follows.  In Section 2, we briefly review some basic notations and facts on vertex operator algebras. In Section 3, we recall the definition of fermionic vertex superalgebras and a construction of the affine VOA $L_{\wo_{n}}(m,0)$ using the fermionic vertex superalgebras. In Section 4, we discuss and prove our main results.

\section{Preliminaries}
\setcounter{equation}{0}
%\subsection{}
Let $V=(V,Y,{\bf 1},\omega)$ be a vertex operator algebra \cite{B},
\cite{FLM}, \cite{LL}. We review various notions of $V$-modules and the definition of rational vertex operator algebras and some basic facts 
\cite{FLM}, \cite{Z}, \cite{DLM3}, \cite{LL}.

\begin{defn} A weak $V$-module is a vector space $M$ equipped
with a linear map
$$
\begin{array}{ll}
Y_M: & V \rightarrow {\rm End}(M)[[z,z^{-1}]]\\
 & v \mapsto Y_M(v,z)=\sum_{n \in \Z}v_n z^{-n-1},\ \ v_n \in {\rm End}(M)
\end{array}
$$
satisfying the following:

1) $v_nw=0$ for $n>>0$ where $v \in V$ and $w \in M$;

2) $Y_M( {\textbf 1},z)=\on{id}_M$;

3) the Jacobi identity holds:
\begin{eqnarray}
& &z_0^{-1}\de \left({z_1 - z_2 \over
z_0}\right)Y_M(u,z_1)Y_M(v,z_2)-
z_0^{-1} \de \left({z_2- z_1 \over -z_0}\right)Y_M(v,z_2)Y_M(u,z_1) \nonumber \\
& &\ \ \ \ \ \ \ \ \ \ =z_2^{-1} \de \left({z_1- z_0 \over
z_2}\right)Y_M(Y(u,z_0)v,z_2).
\end{eqnarray}
\end{defn}

%admissible

\begin{defn}
An admissible $V$-module is a weak $V$-module  which carries a
$\Z_+$-grading $M=\bigoplus_{n \in \Z_+} M(n)$, such that if $v \in
V_r$ then $v_m M(n) \subseteq M(n+r-m-1).$
\end{defn}

\begin{defn}
An ordinary $V$-module is a weak $V$-module which carries a
$\C$-grading $M=\bigoplus_{\l \in \C} M_{\l}$ such that

1) $\dim(M_{\l})< \infty;$

2) $M_{\l+n}=0$ for fixed $\l$ and $n<<0;$

3) $L(0)w=\l w=\wt(w) w$ for $w \in M_{\l}$ where $L(0)$ is the
component operator of $Y_M(\omega,z)=\sum_{n\in\Z}L(n)z^{-n-2}.$
\end{defn}

\begin{remark} \ It is easy to see that an ordinary $V$-module is admissible. If $W$  is an
ordinary $V$-module, we simply call $W$ a $V$-module.
\end{remark}

We call a vertex operator algebra rational if the admissible module
category is semi-simple. We have the following result from
\cite{DLM3} (also see \cite{Z}).

\begin{theorem}\label{tt2.1}
If $V$ is a  rational vertex operator algebra, then $V$ has finitely
many irreducible admissible modules up to isomorphism and every
irreducible admissible $V$-module is ordinary.
\end{theorem}
\section{Fermionic vertex superalgebras}

In this section, we recall the basic fact on infinite-dimensional Clifford algebras and the associated vertex operator superalgebras \cite{Fei, FFR,FF1,Fr, K1,K2,KWak1,KW,AP,Li1}. Let $m\in\Z_{+}$. The Clifford algebra ${\mathcal Cl}_{2m}$ is a complex associative algebra generated by $\psi_i^{\pm}(r), 1\leq i\leq m$, $r\in\Z+\frac{1}{2}$, satisfying the non-trivial relations
$$
[\psi_{i}^{\pm}(r),\psi_{j}^{\mp}(s)]_+=\psi_{i}^{\pm}(r)\psi_{j}^{\mp}(s)+\psi_{j}^{\mp}(s)\psi_{i}^{\pm}(r)=\delta_{r+s,0}\delta_{ij},
$$
where $1\leq i,j\leq m$, $r,s\in\Z+\frac{1}{2}$.

Let ${\mathcal F}_{2m}$ be the irreducible ${\mathcal Cl}_{2m}$-module generated by the cyclic vector ${\bf 1}$ such that
$$
\psi^{\pm}_{i}(r){\bf 1}=0, \ {\rm for} \ r>0, \ 1\leq i\leq m.
$$
Define the  fields $\psi_{i}^{\pm}(z)$, $1\leq i\leq m$, on ${\mathcal F}_{2m}$ by
$$
\psi_{i}^{\pm}(z)=\sum\limits_{r\in\Z}\psi^{\pm}_{i}(r+\frac{1}{2})z^{-r-1}.
$$
The fields $\psi^{\pm}_{i}(z)$, $1\leq i\leq m$, generate ${\mathcal F}_{2m}$, which has a unique structure  of a simple vertex superalgebra \cite{FFR,K2,KW,Li1,AP}.

For $r,k\in\Z_{+}$, let $L_{\wo_{r}}(k,0)$  be the simple vertex operator algebra associated to the integrable module of the affine orthogonal Lie algebra $\wo_r$ with level $k$. The following result is well known \cite{FF1}.
\begin{theorem}\label{t3.1}
Let ${({\mathcal F}_{2m})}^{even}$ be the even part of the vertex superalgebra
${\mathcal F}_{2m}$. Then
$$
{({\mathcal F}_{2})}^{even}\cong V_{\sqrt{2}A_1}$$
and
$$
{({\mathcal F}_{2m})}^{even}\cong L_{\wo_{2m}}(1,0),$$
for $m\geq 2$.
\end{theorem}

Let $m\in\Z_{+}$. We  now consider the Clifford algebra ${\mathcal Cl}_{2m+1}$ generated by
$$
\psi_{i}^{\pm}(r), \ \psi_{2m+1}(r), \ r\in\Z+\frac{1}{2}, \ 1\leq i\leq m
$$
with the  non-trivial relations
$$
[{\psi_{i}^{\pm}(r),\psi_{k}^{\mp}(s)}]_{+}=\delta_{ik}\delta_{r+s,0}, \
$$
$$
[\psi_{2m+1}(r),\psi_{2m+1}(s)]_+=\delta_{r+s,0},
$$
where $1\leq i,k\leq m$, $r,s\in\Z+\frac{1}{2}$. Let ${\mathcal F}_{2m+1}$ be the irreducible ${\mathcal Cl}_{2m+1}$-module generated by the cyclic vector ${\bf 1}$ such that
$$
\psi^{\pm}_{i}(r){\bf 1}=\psi_{2m+1}(r){\bf 1}=0, \ {\rm for} \ r>0, \ 1\leq i\leq m.
$$
Define the following fields on ${\mathcal F}_{2m+1}$ as follows:
$$
\psi_{i}^{\pm}(z)=\sum\limits_{r\in\Z}\psi^{\pm}_{i}(r+\frac{1}{2})z^{-r-1},
$$
$$
\psi_{2m+1}(z)=\sum\limits_{r\in\Z}\psi_{2m+1}(r+\frac{1}{2})z^{-r-1}.
$$
The fields $\psi^{\pm}_{i}(z)$, $\psi_{2m+1}(z)$, $1\leq i\leq m$, generate ${\mathcal F}_{2m+1}$, which has a unique structure  of a simple vertex superalgebra \cite{FFR,K2,KW,Li1,AP}.  By \cite{FF1}, we have the following theorem.
\begin{theorem}
For $m\in \Z_{\geq 2}$, we have
$$
({\mathcal F}_{2m+1})^{even}\cong L_{\wo_{2m+1}}(1,0).
$$
\end{theorem}

\section{Level-rank duality for affine vertex operator algebras of type $B$ and $D$}

In this section, we will study the tensor decomposition of affine vertex operator algebras of type $D$ and $B$.

\subsection{Type $D$}
For $m\geq 2$, let $L_{\wo_{2m}}(1,0)$ be the rational simple vertex operator algebra associated with the integrable highest weight module of the affine Lie algebra $\wo_{2m}$ with level 1. For $n\in\Z_{\geq 2}$,
denote
 $$
 V=L_{\wo_{2m}}(1,0)^{\otimes n}.$$
 Then
 $$
 U=L_{\wo_{2m}}(n,0)
  $$
can be diagonally imbedded into $V$ as a vertex operator subalgebra. Our main aim in this section is to study the commutant $C_V(U)$ of $U$ in $V$. The proof of the following lemma is quite similar to the proof of Lemma 3.1 and Lemma 3.5 in \cite{JL2}.
\begin{lem}\label{dual-pair-2}
The commutant $C_V(U)$ is a simple vertex operator subalgebra of $V$, and $(U, C_V(U))$ is a duality pair in $V$, i.e., $C_V(C_V(U)) =U$.
\end{lem}

 \subsubsection{The case $m\geq 4$ and $n=2$} We define the following lattice
 $$
 R_{2m}=\Z x_{1}\oplus \cdots\oplus \Z x_{2m}
 $$
where $(x_{i},x_j)=\delta_{ij},   \ 1\leq i,j\leq 2m.$

  Let $V_{R_{2m}}$ be the associated lattice vertex superalgebra with the following $2$-cocycle $\varepsilon$
  $$
  \varepsilon(x_i, x_j)=\left\{\begin{array}{ll} 1,  & \ {\rm if} \ i\leq j,\\
  -1, & \ {\rm if}   \ i>j.
  \end{array}
  \right.
  $$
Let $$L_1=\bigoplus_{i=1}^{m-1}\Z(x_i\pm x_{i+1})\quad \text{ and }\quad    \ L_2=\bigoplus_{i=1}^{m-1}\Z(x_{m+i}\pm  x_{m+i+1})$$
be two sublattices of $R_{2m}$. Then $L_1\cong L_2\cong D_m$ and
$$
V_{L_1}\cong V_{L_2} \cong L_{\wo_{2m}}(1,0),
$$
and $L_{\wo_{2m}}(2,0)$ can be naturally regarded as a  vertex operator subalgebra of $V_{L_1}\otimes V_{L_2}$. Set
$$
V=V_{L_1}\otimes V_{L_2} \quad \text{ and } \quad  U=L_{\wo_{2m}}(2,0).
$$
Let  $\widetilde{\mathcal{D}}_m$ be the  sublattice of $R_{2m}$ $\Z$-linearly spanned  by $$\{x_i+x_{i+1}+x_{m+i}+x_{m+i+1},
x_i-x_{i+1}+x_{m+i}-x_{m+i+1}\mid 1\leq i\leq m-1\},$$
and let $\mathcal{D}_m$ be the sublattice of $R_{2m}$ $\Z$-linearly spanned  by
$$\{x_i+x_j-x_{m+i} -x_{m+j},
x_i-x_j-x_{m+i}+x_{m+j}\mid 1\leq i<j\leq m\}.$$
Notice that $\widetilde{\mathcal{D}}_m \cong \mathcal{D}_m\cong \sqrt{2}D_m$ as lattices.
By a similar argument as in \cite{LY}, it is easy to verify that
$$
V_{\widetilde{\mathcal{D}}_m}\subseteq U=L_{\wo_{2m}}(2,0).
$$
Moreover, we have the following lemma.
\begin{lem}\label{dual-pair-1}
(1) $(V_{\widetilde{\D}_m}, V_{\D_m})$ is a duality pair in $V$.

(2) $(C_V(U), K(\frak{so}_{2m},2))$ is a duality pair in $V_{\D_m}$, where $K(\frak{so}_{2m},2)$
 is the commutant of  $V_{\widetilde{\D}_m}$ in $L_{\wo_{2m}}(2,0)$, which  is called a parafermion vertex operator algebra.
 \end{lem}
\begin{proof}
Notice that $\widetilde{\D}_m$ and $\D_m$ are orthogonal to each other. It means that the associated lattice vertex operator algebras
$V_{\widetilde{\D}_m}$ and $V_{\D_m}$ are commutant to each other. The rest of the  proof of (1) is similar to that of Lemma 3.2 in \cite{JL2}.

By Lemma \ref{dual-pair-2}, we know that $(C_V(U), L_{\wo_{2m}}(2,0))$ is a duality pair in $V$. Notice also that
$(V_{\widetilde{\D}_m}, K(\frak{so}_{2m},2))$ is a duality pair in $L_{\wo_{2m}}(2,0)$ ( see \cite{DLY2} ). Then (2) follows from (1) and the reciprocity law established in Theorem 5.4 of \cite{JL2}.
\end{proof}
Denote
$$
\al^i=x_i-x_{m+i}, \ 1\leq i\leq m.$$
Then $\{\al^1-\al^2, \al^2-\al^3, \cdots, \al^{m-1}-\al^m, \al^{m-1}+\al^{m}\}$ is a basis of $\D_m\cong \sqrt{2}D_m$.
Consider
$$
V_{A_1^m}=V_{\Z\al^1}\otimes V_{\Z\al^2}\otimes \cdots \otimes V_{\Z\al^m}.
$$
Define vertex operator algebra automorphisms of $V_{A_1^m}$ as follows:
$$
\tau: \al^i(-1){\bf 1}\mapsto e^{\al^i}+e^{-\al^i}, \
e^{\al^i}+e^{-\al^i}\mapsto \al^i(-1){\bf 1}, \  e^{\al^i}-e^{-\al^i}
\mapsto -(e^{\al^i}-e^{-\al^i}),$$
$$
u^1\otimes u^2\otimes \cdots \otimes u^m \mapsto \tau(u^1)\otimes \tau(u^2)\otimes \cdots \otimes \tau(u^m);
$$
$$
\theta: \al^i(-1){\bf 1}\mapsto \al^i(-1){\bf 1}, \ e^{\al^i}\mapsto e^{\al^i}, \
e^{-\al^i}\mapsto e^{-\al^i}, \ i \ {\rm is} \ {\rm odd}, $$
$$
\al^i(-1){\bf 1}\mapsto -\al^i(-1){\bf 1}, \  e^{\al^i}\mapsto e^{-\al^i}, \
e^{-\al^i}\mapsto e^{\al^i}, \ i \ {\rm is} \ {\rm even}, $$
$$
u^1\otimes u^2\otimes \cdots \otimes u^m \mapsto \theta(u^1)\otimes \theta(u^2)\otimes \cdots \otimes \theta(u^m);
$$
$$
\sigma: \al^i(-1){\bf 1}\mapsto -\al^i(-1){\bf 1}, \  e^{\al^i}\mapsto e^{-\al^i}, \
e^{-\al^i}\mapsto e^{\al^i}, $$
$$ u^1\otimes u^2\otimes \cdots \otimes u^m \mapsto \sigma(u^1)\otimes \sigma(u^2)\otimes \cdots \otimes \sigma(u^m);
$$
where $u^i\in V_{\Z\al^i}$, $1\leq i\leq m$. Then we have
\begin{lem}[\cite{DLY1}]\label{iso1}
$\theta\tau(V_{\sqrt{2}D_m})=V_{A_1^m}^{\sigma}$.
\end{lem}
It can be checked directly that, for $1\leq i\leq m-1$,
\begin{equation}\label{4eq2}
\theta\tau[\frac{1}{4}(\al^i-\al^{i+1})(-1)^2{\bf 1}+(e^{\al^i-\al^{i+1}}+e^{-\al^i+\al^{i+1}})]= \frac{1}{4}(\al^i-\al^{i+1})(-1)^2{\bf 1}-(e^{\al^i-\al^{i+1}}+e^{-\al^i+\al^{i+1}}),
\end{equation}
\begin{equation}\label{4eq3}
 \theta\tau[\frac{1}{4}(\al^i+\al^{i+1})(-1)^2{\bf 1}+(e^{\al^i+\al^{i+1}}+e^{-\al^i-\al^{i+1}})]
=  \frac{1}{4}(\al^i-\al^{i+1})(-1)^2{\bf 1}+(e^{\al^i-\al^{i+1}}+e^{-\al^i+\al^{i+1}}).
\end{equation}
The following Lemma comes from \cite {La}, \cite{JL2}, \cite{LS}, and \cite{JL1}.
\begin{lem}\label{dual-pair-3}
(1) $(L_{\widehat{\frak{sl}}_2}(m,0), K(\frak{sl}_m,2))$ is a duality pair in $V_{A_1^m}$.

(2) $K(\frak{sl}_m,2)$ is  generated by
$$
\omega^{ij}=\frac{1}{16}(\al^i-\al^{i+1})(-1)^2{\bf 1}-\frac{1}{4}(e^{\al^i-\al^{i+1}}+e^{-\al^i+\al^{i+1}}),
\ 1\leq i<j\leq m.$$

(3) $K(\frak{sl}_m,2)$ is rational.
\qed
\end{lem}
By (2) of Lemma \ref{dual-pair-3}, we have
\begin{equation}\label{4eq1}
  K(\frak{sl}_m,2)\subseteq V_{A_1^m}^{\sigma}.
  \end{equation}

Then by (1) of Lemma \ref{dual-pair-3}, we have the following lemma.
\begin{lem}\label{dual-pair-4}
$
(K(\frak{sl}_m,2), L_{\widehat{\frak{sl}}_2}(m,0)^{\sigma})
$ is a duality pair in $V_{A_1^m}^{\sigma}$.
\end{lem}
Let
$$
\gamma=\al^1+\al^2+\cdots+\al^m.
$$
We are now in a position to state the following level-rank duality.
\begin{theorem}\label{duality-1}
(1) $C_{L_{\wo_{2m}}(1,0)^{\otimes 2}}(L_{\wo_{2m}}(2,0)) \cong V_{\Z\gamma}^{\sigma}$.

(2) $C_{L_{\wo_{2m}}(1,0)^{\otimes 2}}(L_{\wo_{2m}}(2,0))\cong C_{ L_{\widehat{\frak{sl}}_2}(m,0)^{\sigma}}(K(\frak{sl}_2,m)^{\sigma})$.

(3) $C_{K(\frak{so}_{2m}, 2)}(K(\frak{sl}_{m},2))\cong K(\frak{sl}_2,m)^{\sigma}$.
\end{theorem}
\begin{proof}
By (1) of Lemma \ref{dual-pair-3}, $(L_{\widehat{\frak{sl}}_2}(m,0), K(\frak{sl}_m,2))$ is a duality pair in $V_{A_1^m}$. Since $(V_{\Z\gamma}, K(\frak{sl}_2,m))$ is a duality pair in $L_{\widehat{\frak{sl}}_2}(m,0)$, we deduce that
$$
C_{V_{A_1^m}}(K(\frak{sl}_{m},2)\otimes K(\frak{sl}_{2},m))=V_{\Z\gamma}.$$
By (\ref{4eq1}) and the fact that $K(\frak{sl}_{m},2)\otimes K(\frak{sl}_{2},m)$ and $K(\frak{sl}_{m},2)\otimes K(\frak{sl}_{2},m)^{\sigma}$ have the same conformal vector, we have
\begin{equation}\label{4eq4}
C_{V_{A_1^m}^{\sigma}}(K(\frak{sl}_{m},2)\otimes K(\frak{sl}_{2},m)^{\sigma})=V_{\Z\gamma}^{\sigma}.
\end{equation}
By Lemma \ref{iso1},  $\theta\tau(V_{\sqrt{2}D_m})=V_{A_1^m}^{\sigma}$. Notice  that $K(\frak{so}_{2m},2)$ is generated by $$
\frac{1}{4}(\al^i-\al^{i+1})(-1)^2{\bf 1}+(e^{\al^i-\al^{i+1}}+e^{-\al^i+\al^{i+1}}), \
\frac{1}{4}(\al^i+\al^{i+1})(-1)^2{\bf 1}+(e^{\al^i+\al^{i+1}}+e^{-\al^i-\al^{i+1}}),
$$
where $ 1\leq i\leq m-1$. By (\ref{4eq2}) and (\ref{4eq3}),
$$\theta\tau(K(\frak{so}_{2m},2))\subseteq C_{V_{A_1^m}^{\sigma}}(V_{\Z\gamma}^{\sigma}).$$
Then by (2) of Lemma \ref{dual-pair-1}, we have
\begin{equation}\label{4eq5}
V_{\Z\gamma}^{\sigma}\subseteq C_{V_{A_1^m}^{\sigma}}(\theta\tau(K(\frak{so}_{2m},2)))
=\theta\tau(C_{V}(U)).
\end{equation}
On the other hand, it can be checked directly that $\theta\tau(K(\frak{so}_{2m},2))$ and $K(\frak{sl}_{m},2)\otimes K(\frak{sl}_{2},m)^{\sigma}$ have the same conformal vector. Then by (\ref{4eq4}), we have
$$\theta\tau(C_{V}(U))\subseteq V_{\Z\gamma}^{\sigma}.$$
Together with (\ref{4eq5}), we deduce that
$$\theta\tau(C_{V}(U))= V_{\Z\gamma}^{\sigma},$$
proving (1) and (2). Since $(\theta\tau(K(\frak{so}_{2m},2)), \theta\tau(C_{V}(U)))$ is a duality pair in $V_{A_1^m}^{\sigma}$, we have
$$
K(\frak{sl}_{m},2)\otimes K(\frak{sl}_{2},m)^{\sigma}\subseteq \theta\tau(K(\frak{so}_{2m},2)).$$
Then
$$
 K(\frak{sl}_{2},m)^{\sigma}\subseteq C_{\theta\tau(K(\frak{so}_{2m}, 2))}(K(\frak{sl}_{m},2)).$$
 Notice that
 $$
 C_{L_{\widehat{\frak{sl}}_2}(m,0)^{\sigma}}(V_{\Z\gamma}^{\sigma})=K(\frak{sl}_2,m)^{\sigma}
  $$
 and $$C_{V_{A_1^m}^{\sigma}}(K(\frak{sl}_{m},2))=L_{\widehat{\frak{sl}}_2}(m,0)^{\sigma}.$$
 So
 $$C_{\theta\tau(K(\frak{so}_{2m}, 2))}(K(\frak{sl}_{m},2))\subseteq L_{\widehat{\frak{sl}}_2}(m,0)^{\sigma}.
 $$
 Since
 $
 C_{\theta\tau(K(\frak{so}_{2m}, 2))}(K(\frak{sl}_{m},2))$ commutes with  $V_{\Z\gamma}^{\sigma}$,  it  follows that
 $$
  C_{\theta\tau(K(\frak{so}_{2m}, 2))}(K(\frak{sl}_{m},2)) =K(\frak{sl}_2,m)^{\sigma}.$$
Therefore (3) holds.
\end{proof}
\begin{remark}
For the proof of (1) in  Theorem \ref{duality-1}, one can also refer to \cite{AP}.
\end{remark}

\subsubsection{The case $m\in\Z_{\geq 2}$ and $n\in\Z_{\geq 3}$}
Let ${\mathcal Cl}_{2mn}$ be the Clifford algebra generated by $\psi^{\pm}_{ij}(r)$, $1\leq i\leq m, 1\leq j\leq n$, $r\in\Z+\frac{1}{2}$, with the non-trivial relations
\begin{equation}\label{4eq5_1}
[\psi_{ij}^{\pm}(r),\psi_{kl}^{\mp}(s)]_+
=\psi_{ij}^{\pm}(r)\psi_{kl}^{\mp}(s)+\psi_{kl}^{\mp}(s)\psi_{ij}^{\pm}(r) =\delta_{ik}\delta_{jl}\delta_{r+s,0},
\end{equation}
where $1\leq i,k\leq m, 1\leq j,l\leq n$, $r,s\in\Z+\frac{1}{2}$.
 Let ${\mathcal F}_{2mn}$ be the irreducible ${\mathcal Cl}_{2mn}$-module generated by the cyclic vector ${\bf 1}$ such that
$$
\psi^{\pm}_{ij}(r){\bf 1}=0, \ {\rm for} \ r>0, \ 1\leq i\leq m, \ 1\leq j\leq n.
$$
Then by Theorem \ref{t3.1}, ${\mathcal F}^{even}_{2mn}\cong L_{\wo_{2mn}}(1,0)$. Obviously ${\mathcal F}^{even}_{2mn}$
is generated by
$$
\psi^{\pm}_{ki}(-\frac{1}{2})\psi^{\pm}_{lj}(-\frac{1}{2}){\bf 1}, \  \  \psi^{\pm}_{ki}(-\frac{1}{2})\psi^{\mp}_{lj}(-\frac{1}{2}){\bf 1}, \ 1\leq k,l\leq m, \ 1\leq i,j\leq n.$$

 Define the  fields on ${\mathcal F}_{2mn}$ as follows
$$
\psi_{ij}^{\pm}(z)=\sum\limits_{r\in\Z}\psi^{\pm}_{ij}(r+\frac{1}{2})z^{-r-1}.
$$
 For $1\leq j\leq n$, let ${\mathcal F}_{2m,j}$ (resp. ${\mathcal F}_{i,n}$, for $1\leq i\leq m$) be the  subalgebra of ${\mathcal F}_{2mn}$ generated by the fields $\psi^{\pm}_{ij}(z)$, $1\leq i\leq m$ (resp. $\psi^{\pm}_{ij}(z)$, $1\leq j\leq n$).
Then ${\mathcal F}_{2m,j}^{even}\cong L_{\wo_{2m}}(1,0)$.  We denote ${\mathcal F}_{2m,j}^{even}$ by
$L^{(j)}_{\wo_{2m}}(1,0)$. Then we have
$$
V=L_{\wo_{2m}}^{(1)}(1,0)\otimes L_{\wo_{2m}}^{(2)}(1,0)\otimes \cdots\otimes L_{\wo_{2m}}^{(n)}(1,0)=L_{\wo_{2m}}(1,0)^{\otimes n}\subseteq L_{\wo_{2mn}}(1,0).$$
It is obvious that   $L_{\wo_{2m}}(n,0)$ can be diagonally embedded into $V$ as a vertex operator subalgebra of $V$, and
$$
\sum\limits_{j=1}^{n}\psi^{\pm}_{kj}(-\frac{1}{2})\psi^{\pm}_{lj}(-\frac{1}{2}){\bf 1}, \  \  \sum\limits_{j=1}^{n}\psi^{\pm}_{kj}(-\frac{1}{2})\psi^{\mp}_{lj}(-\frac{1}{2}){\bf 1}, \ 1\leq k,l\leq m$$
 are generators of $L_{\wo_{2m}}(n,0)$. We denote this diagonally embedded vertex operator subalgebra of $V$ by $U$ and $U\cong L_{\wo_{2m}}(n,0)$.
Specifically, let
$$
\alpha^{\vee}_{kj}=\psi^{+}_{kj}(-\frac{1}{2})\psi^{-}_{kj}(-\frac{1}{2}){\bf 1}-\psi^{+}_{k+1,j}(-\frac{1}{2})\psi^{-}_{k+1,j}(-\frac{1}{2}){\bf 1},$$
$$
e_{kj}=\sqrt{-1}\psi^{+}_{kj}(-\frac{1}{2})\psi^{-}_{k+1,j}(-\frac{1}{2}){\bf 1}, \
f_{kj}=\sqrt{-1}\psi^{-}_{kj}(-\frac{1}{2})\psi^{+}_{k+1,j}(-\frac{1}{2}){\bf 1},$$
for $1\leq k\leq m-1$ and
$$
\alpha^{\vee}_{mj}=\psi^{+}_{m-1,j}(-\frac{1}{2})\psi^{-}_{m-1,j}(-\frac{1}{2}){\bf 1}+\psi^{+}_{mj}(-\frac{1}{2})\psi^{-}_{mj}(-\frac{1}{2}){\bf 1},$$
$$
e_{mj}=\sqrt{-1}\psi^{+}_{m-1,j}(-\frac{1}{2})\psi^{+}_{mj}(-\frac{1}{2}){\bf 1}, \
f_{mj}=\sqrt{-1}\psi^{-}_{m-1,j}(-\frac{1}{2})\psi^{-}_{mj}(-\frac{1}{2}){\bf 1}.$$
Then it can be checked directly that $\{\alpha_{kj}^{\vee}, e_{kj},f_{kj}, 1\leq k\leq m\}$ are Chevalley generators of the simple Lie algebra $\frak{so}_{2m}(\C)$.

Set
$$
\psi_{kj}(-\frac{1}{2}){\bf 1}=\frac{1}{\sqrt{2}}(\psi^{+}_{kj}(-\frac{1}{2}){\bf 1}+\psi^{-}_{kj}(-\frac{1}{2}){\bf 1}),$$
$$
\psi_{m+k,j}(-\frac{1}{2}){\bf 1}=\frac{-\sqrt{-1}}{\sqrt{2}}(\psi^{+}_{kj}(-\frac{1}{2}){\bf 1}-\psi^{-}_{kj}(-\frac{1}{2}){\bf 1}),
$$
for $1\leq k\leq m$, $1\leq j\leq n$. Then
$$
\psi^{+}_{kj}(-\frac{1}{2}){\bf 1}=\frac{1}{\sqrt{2}}(\psi_{kj}(-\frac{1}{2}){\bf 1}+\sqrt{-1}\psi_{m+k,j}(-\frac{1}{2}){\bf 1}),$$
$$
\psi^{-}_{kj}(-\frac{1}{2}){\bf 1}=\frac{1}{\sqrt{2}}(\psi_{kj}(-\frac{1}{2}){\bf 1}
-\sqrt{-1}\psi_{m+k,j}(-\frac{1}{2}){\bf 1}),$$
and
\begin{equation}\label{4eq6}
[\psi_{ki}(r),\psi_{lj}(s)]_+=\delta_{ij}\delta_{kl}\delta_{r+s,0}.
\end{equation}
It is easy to see that for $1\leq k\leq m-1$,
$$
\alpha^{\vee}_{kj}=\sqrt{-1}(\psi_{m+k,j}(-\frac{1}{2})\psi_{kj}(-\frac{1}{2}){\bf 1}-\psi_{m+k+1,j}(-\frac{1}{2})\psi_{k+1,j}(-\frac{1}{2}){\bf 1}),$$
$$
\begin{array}{ll}
e_{kj}=& \dfrac{1}{2}[(\psi_{kj}(-\frac{1}{2})\psi_{m+k+1,j}(-\frac{1}{2}){\bf 1}+\psi_{k+1,j}(-\frac{1}{2})\psi_{m+k,j}(-\frac{1}{2}){\bf 1})\\
& +\sqrt{-1}(\psi_{kj}(-\frac{1}{2})\psi_{k+1,j}(-\frac{1}{2}){\bf 1}+\psi_{m+k,j}(-\frac{1}{2})\psi_{m+k+1,j}(-\frac{1}{2}){\bf 1})],
\end{array}
$$$$
\begin{array}{ll}
f_{kj}=& \dfrac{1}{2}[-(\psi_{kj}(-\frac{1}{2})\psi_{m+k+1,j}(-\frac{1}{2}){\bf 1}+\psi_{k+1,j}(-\frac{1}{2})\psi_{m+k,j}(-\frac{1}{2}){\bf 1})\\
& +\sqrt{-1}(\psi_{kj}(-\frac{1}{2})\psi_{k+1,j}(-\frac{1}{2}){\bf 1}+\psi_{m+k,j}(-\frac{1}{2})\psi_{m+k+1,j}(-\frac{1}{2}){\bf 1})],
\end{array}$$
 and
$$
\alpha^{\vee}_{mj}=\sqrt{-1}(\psi_{2m-1,j}(-\frac{1}{2})\psi_{m-1,j}(-\frac{1}{2}){\bf 1}+\psi_{2m,j}(-\frac{1}{2})\psi_{mj}(-\frac{1}{2}){\bf 1}),$$
$$
\begin{array}{ll}
e_{mj}=& \dfrac{1}{2}[(-\psi_{m-1,j}(-\frac{1}{2})\psi_{2m,j}(-\frac{1}{2}){\bf 1}+\psi_{mj}(-\frac{1}{2})\psi_{2m-1,j}(-\frac{1}{2}){\bf 1})\\
& +\sqrt{-1}(\psi_{m-1,j}(-\frac{1}{2})\psi_{mj}(-\frac{1}{2}){\bf 1}-\psi_{2m-1,j}(-\frac{1}{2})\psi_{2m,j}(-\frac{1}{2}){\bf 1})],
\end{array}
$$$$
\begin{array}{ll}
f_{mj}=& \dfrac{1}{2}[(\psi_{m-1,j}(-\frac{1}{2})\psi_{2m,j}(-\frac{1}{2}){\bf 1}-\psi_{mj}(-\frac{1}{2})\psi_{2m-1,j}(-\frac{1}{2}){\bf 1})\\
& +\sqrt{-1}(\psi_{m-1,j}(-\frac{1}{2})\psi_{mj}(-\frac{1}{2}){\bf 1}-\psi_{2m-1,j}(-\frac{1}{2})\psi_{2m,j}(-\frac{1}{2}){\bf 1})].
\end{array}$$

It is also easy to check that ( see also \cite{FKRW} )
$$
[Y(\sum\limits_{k=1}^{2m}\psi_{ki}(-\frac{1}{2})\psi_{kj}(-\frac{1}{2}){\bf 1},z),
Y(\sum\limits_{l=1}^{n}\psi_{rl}(-\frac{1}{2})\psi_{sl}(-\frac{1}{2}){\bf 1},w)]=0,
$$
 for $1\leq i,j\leq n$, $1\leq r,s\leq 2m$.  Moreover, the vertex operator subalgebra of $L_{\wo_{2mn}}(1,0)$ generated by $\sum\limits_{k=1}^{2m}\psi_{ki}(-\frac{1}{2})\psi_{kj}(-\frac{1}{2}){\bf 1}$, $1\leq i,j\leq n$,  is isomorphic to $L_{\wo_{n}}(2m,0)$.  We denote this vertex operator subalgebra by $X\ (\cong L_{\wo_{n}}(2m,0))$.

 Denote by $C_{L_{\wo_{2mn}}(1,0)}(L_{\wo_{2m}}(n,0))$ the  commutant of $U=L_{\wo_{2m}}(n,0)$ in $L_{\wo_{2mn}}(1,0)$.  Then we have the following lemma.
  \begin{lem}\label{commutant4}
 $X  \subseteq  C_{L_{\wo_{2mn}}(1,0)}(L_{\wo_{2m}}(n,0))$.
 \end{lem}
 Recall that  the vertex operator algebra $L_{\wo_{2mn}}(1,0)$ is generated by $\psi_{kj}(-\frac{1}{2})\psi_{rs}(-\frac{1}{2}){\bf 1}$, $1\leq k,r\leq 2m$, $1\leq j,s\leq n$.  For $1\leq i\leq n$, we define the vertex operator algebra automorphism $\sigma_i$ of $L_{\wo_{2mn}}(1,0)$ by
\begin{equation}\label{aeq3}
\sigma_{i}(\psi_{kj}(-\frac{1}{2})\psi_{rs}(-\frac{1}{2}){\bf 1})=(-1)^{\delta_{ij}+\delta_{is}}\psi_{kj}(-\frac{1}{2})\psi_{rs}(-\frac{1}{2}){\bf 1}.
\end{equation}
Then $$\sigma_{i}^2={\rm id},$$
and
for
$1\leq j\leq n$,  $u=\psi_{k_1i_1}(-m_1-\frac{1}{2})\cdots \psi_{k_{2r}i_{2r}}(-m_{2r}-\frac{1}{2}){\bf 1}\in L_{\wo_{2mn}}(1,0)$, where $1\leq k_1,\cdots,k_{2r}\leq 2m$, $1\leq i_1,\cdots,i_{2r}\leq n$, $m_{1}, \cdots,m_{2r}\in\N$,
$$
\sigma_ju=(-1)^{\delta_{j{i_1}}+\cdots+\delta_{ji_{2r}}}u.
$$

Denote by $G$ the automorphism
group of $L_{\wo_{2mn}}(1,0)$ generated by $\{\sigma_{i}$, $1\leq i\leq n\}$.  It is obvious that $G$ is an abelian group. Set
$$
L_{\wo_{2mn}}(1,0)^{G}=\{v\in L_{\wo_{2mn}}(1,0) \mid  g(v)=v, \ g\in G\}.$$
Let $V=L_{\wo_{2m}}^{(1)}(1,0)\otimes L_{\wo_{2m}}^{(2)}(1,0)\otimes \cdots\otimes L_{\wo_{2m}}^{(n)}(1,0)\cong L_{\wo_{2m}}(1,0)^{\otimes n}$ and $U=L_{\wo_{2m}}(n,0)$ be defined as above. We have the following lemma.
\begin{lem}\label{auto1} $U=L_{\wo_{2m}}(n,0)\subseteq L_{\wo_{2mn}}(1,0)^{G}$ and
$$L_{\wo_{2mn}}(1,0)^{G}=V=L_{\wo_{2m}}^{(1)}(1,0)\otimes L_{\wo_{2m}}^{(2)}(1,0)\otimes \cdots\otimes L_{\wo_{2m}}^{(n)}(1,0). $$
\end{lem}
\begin{proof}
Recall that for $1\leq j\leq n$, $L^{(j)}_{\wo_{2m}}(1,0)$ is generated by $\psi_{kj}(-\frac{1}{2})\psi_{lj}(-\frac{1}{2}){\bf 1}$, $1\leq k,l\leq 2m$.  By the definition of $\sigma_i$, for any $1\leq i\leq n$
$$
\sigma_i(\psi_{kj}(-\frac{1}{2})\psi_{lj}(-\frac{1}{2}){\bf 1})=\psi_{kj}(-\frac{1}{2})\psi_{lj}(-\frac{1}{2}){\bf 1},
$$
So
$L_{\wo_{2m}}^{(j)}(1,0)\subseteq L_{\wo_{2mn}}(1,0)^{G}$. This implies that $V\subseteq L_{\wo_{2mn}}(1,0)^{G}$. In general, by the definition of $\sigma_j$, $1\leq j\leq n$, we have
$$
\sigma_ju=(-1)^{\delta_{j{i_1}}+\cdots+\delta_{ji_{2r}}}u,
$$
for
$u=\psi_{k_1i_1}(-m_1-\frac{1}{2})\cdots \psi_{k_{2r}i_{2r}}(-m_{2r}-\frac{1}{2}){\bf 1}\in L_{\wo_{2mn}}(1,0)$, where $1\leq k_1,\cdots,k_{2r}\leq 2m$, $1\leq i_1,\cdots,i_{2r}\leq n$, $m_{1}, \cdots,m_{2r}\in\N$.
Then for each $j$, $\sigma_ju=u$  if and only if the number of $i_k$'s which $i_k=j$ is even. This means that $u\in V$.
\end{proof}

 \subsubsection{The case $ n=2N$ and $N\geq 2$} In this subsection,  we assume that $m\in\Z_{\geq 2}$, $n=2N\in 2\Z_{+}$ and $N\geq 2$.  Set
 $$
\beta^{\vee}_{k}=\sqrt{-1}\sum\limits_{i=1}^{2m}(\psi_{i,N+k}(-\frac{1}{2})\psi_{ik}(-\frac{1}{2}){\bf 1}-\psi_{i,N+k+1}(-\frac{1}{2})\psi_{i,k+1}(-\frac{1}{2}){\bf 1}),$$
$$
\begin{array}{ll}
e_{k}=& \dfrac{1}{2}\sum\limits_{i=1}^{2m}[(\psi_{ik}(-\frac{1}{2})\psi_{i,N+k+1}(-\frac{1}{2}){\bf 1}+\psi_{i,k+1}(-\frac{1}{2})\psi_{i,N+k}(-\frac{1}{2}){\bf 1})\\
& +\sqrt{-1}(\psi_{ik}(-\frac{1}{2})\psi_{i,k+1}(-\frac{1}{2}){\bf 1}+\psi_{i,N+k}(-\frac{1}{2})\psi_{i,N+k+1}(-\frac{1}{2}){\bf 1})],
\end{array}
$$$$
\begin{array}{ll}
f_{k}=& \dfrac{1}{2}\sum\limits_{i=1}^{2m}[-(\psi_{ik}(-\frac{1}{2})\psi_{i,N+k+1}(-\frac{1}{2}){\bf 1}+\psi_{i,k+1}(-\frac{1}{2})\psi_{i,N+k}(-\frac{1}{2}){\bf 1})\\
& +\sqrt{-1}(\psi_{ik}(-\frac{1}{2})\psi_{i,k+1}(-\frac{1}{2}){\bf 1}+\psi_{i,N+k}(-\frac{1}{2})\psi_{i,N+k+1}(-\frac{1}{2}){\bf 1})],
\end{array}$$
for $1\leq k\leq N-1$ and
$$
\beta^{\vee}_{N}=\sqrt{-1}\sum\limits_{i=1}^{2m}(\psi_{i,2N-1}(-\frac{1}{2})\psi_{i,N-1}(-\frac{1}{2}){\bf 1}+\psi_{i,2N}(-\frac{1}{2})\psi_{iN}(-\frac{1}{2}){\bf 1}),$$
$$
\begin{array}{ll}
e_{N}=& \dfrac{1}{2}\sum\limits_{i=1}^{2m}[(-\psi_{i,N-1}(-\frac{1}{2})\psi_{i,2N}(-\frac{1}{2}){\bf 1}+\psi_{iN}(-\frac{1}{2})\psi_{i,2N-1}(-\frac{1}{2}){\bf 1})\\
& +\sqrt{-1}(\psi_{i,N-1}(-\frac{1}{2})\psi_{iN}(-\frac{1}{2}){\bf 1}-\psi_{i,2N-1}(-\frac{1}{2})\psi_{i,2N}(-\frac{1}{2}){\bf 1})],
\end{array}
$$$$
\begin{array}{ll}
f_{N}=& \dfrac{1}{2}\sum\limits_{i=1}^{2m}[(\psi_{i,N-1}(-\frac{1}{2})\psi_{i,2N}(-\frac{1}{2}){\bf 1}-\psi_{iN}(-\frac{1}{2})\psi_{i,2N-1}(-\frac{1}{2}){\bf 1})\\
& +\sqrt{-1}(\psi_{i,N-1}(-\frac{1}{2})\psi_{iN}(-\frac{1}{2}){\bf 1}-\psi_{i,2N-1}(-\frac{1}{2})\psi_{i,2N}(-\frac{1}{2}){\bf 1})].
\end{array}$$
Then $\{\beta_{k}^{\vee}, e_{k},f_{k}, 1\leq k\leq N\}$ are Chevalley generators of $\frak{so}_{n}$ and generate the vertex operator subalgebra $X \cong L_{\wo_{n}}(2m,0)$. For $1\leq i<j\leq n$, set
\begin{equation}\label{aeq1}
u^{(ij)}=\prod_{k=1}^{2m}(\psi_{ki}(-\frac{1}{2})+\sqrt{-1}\psi_{kj}(-\frac{1}{2})){\bf 1},
\end{equation}
\begin{equation}\label{aeq2}
v^{(ij)}=\prod_{k=1}^{2m}(\psi_{ki}(-\frac{1}{2})-\sqrt{-1}\psi_{kj}(-\frac{1}{2})){\bf 1}.
\end{equation}
The following lemma can be checked by a direct calculation.
\begin{lem}\label{commutant1}
 For $1\leq i<j\leq n$,
$$
u^{(ij)}, v^{(ij)}\in C_{L_{\wo_{2mn}}(1,0)}(L_{\wo_{2m}}(n,0)).$$
Furthermore, for $1\leq k\leq N$, $r\in\Z_{+}$, we have
$$
\beta_{k}^{\vee}(r)u^{(ij)}=e_k(r)u^{(ij)}=f_k(r)u^{(ij)}=0,  $$$$
\beta_{k}^{\vee}(r)v^{(ij)}=e_k(r)v^{(ij)}=f_k(r)v^{(ij)}=0, $$ and if $N\geq 3$,
$$
\beta_1^{\vee}(0)u^{(1,N+1)}=2mu^{(1,N+1)}, \  \beta_s^{\vee}(0)u^{(1,N+1)}=0, \
\ e_k(0)u^{(1,N+1)}=0,
$$
$$
\beta_1^{\vee}(0)v^{(1,N+1)}=-2mv^{(1,N+1)},  \ \beta_s^{\vee}(0)v^{(1,N+1)}=0,
\ f_k(0)v^{(1,N+1)}=0,
$$
where  $2\leq s\leq N$, $1\leq k\leq N$.

 If $N=2$, we have
$$
\beta_i^{\vee}(0)u^{(1,3)}=2mu^{(1,3)}, \
\ e_i(0)u^{(1,3)}=0,
$$
$$
\beta_i^{\vee}(0)v^{(1,3)}=-2mv^{(1,3)},  \
\ f_i(0)v^{(1,3)}=0,
$$
where $i=1,2$.
\end{lem}
If $N\geq 3$, let $\Lambda_i$, $i=0,1,\cdots,N,$ be  the fundamental weights of the affine Lie algebra $\wo_{n}$. Notice that $L_{\wo_{2mn}}(1,0)$ is an integrable module of $X= L_{\wo_{n}}(2m,0)$. Therefore,   the $L_{\wo_{n}}(2m,0)$-submodule of $L_{\wo_{2mn}}(1,0)$ generated by $u^{(1,N+1)}$ is irreducible and is isomorphic to $L_{\wo_{n}}(2m,2m\Lambda_1)$ by Lemma \ref {commutant1}. We simply denote  the $L_{\wo_{n}}(2m,0)$-submodule generated by $u^{(1,N+1)}$ by $L_{\wo_{n}}(2m,2m\Lambda_1)$.

If $N=2$, note that $\frak{so}_{4}(\C)\cong \frak{sl}_2(\C)\oplus \frak{sl}_2(\C)$. By Lemma \ref{commutant1}, $u^{(1,3)}$ generates an irreducible module of $L_{\widehat{\frak{sl}_2}}(2m,0)\otimes L_{\widehat{\frak{sl}_2}}(2m,0)$ isomorphic to $L_{\widehat{\frak{sl}_2}}(2m,2m)\otimes L_{\widehat{\frak{sl}_2}}(2m,2m)$.

Notice that $U\otimes X \cong L_{\wo_{2m}}(n,0)\otimes L_{\wo_{n}}(2m,0)$ is a full vertex operator subalgebra of $L_{\wo_{2mn}}(1,0)$. From now on, we simply identify $U\otimes X$ with $L_{\wo_{2m}}(n,0)\otimes L_{\wo_{n}}(2m,0)$.    The following lemma comes from \cite{KFPX}.
  \begin{lem}\label{add1}
 (1) If $m\equiv N\equiv 1  \ {\rm mod} \  2$, the simple  vertex operator algebras $W$ such that $L_{\wo_{2m}}(n,0)\otimes L_{\wo_{n}}(2m,0)\subseteq W\subseteq L_{\wo_{2mn}}(1,0)$ are in one to one correspondence with the subgroups of $\Z/2\Z\times \Z/2\Z$.

 (2) The simple vertex operator subalgebras $W$ such that $L_{\wo_{2m}}(n,0)\otimes L_{\wo_{n}}(2m,0)\subseteq W \subseteq L_{\wo_{2mn}}(1,0)$ are in one to one correspondence with the subgroups of $(\Z/2\Z)^3$ if $m\equiv N\equiv 0 \ {\rm mod} \ 2$, and  the subgroups of $\Z/2\Z\times \Z/4\Z$ if   $m\equiv 0 \ {\rm mod} \ 2$ and  $N\equiv 1 \ {\rm mod} \ 2$.
  \end{lem}
  We have the following lemma.
\begin{lem}\label{commutant2} For $m\in\Z_{\geq 2}$ and $n=2N\geq 4$, we have
$$C_{L_{\wo_{2mn}}(1,0)}(L_{\wo_{2m}}(n,0))=L_{\wo_{n}}(2m,0)\oplus L_{\wo_{n}}(2m,2m\Lambda_1),$$
where  we mean $L_{\widehat{\frak{sl}_2}}(m,m)\otimes L_{\widehat{\frak{sl}_2}}(m,m)$ by $L_{\wo_{n}}(m,m\Lambda_1)$ if $n=4$.
\end{lem}
\begin{proof} By Lemma \ref{commutant1} and Lemma  \ref{add1}, we just need to  prove the lemma for the case  $N\geq 3$. By  Lemma \ref{commutant4} and Lemma \ref{commutant1}, we have
$$
L_{\wo_{n}}(2m,0)\oplus L_{\wo_{n}}(2m,2m\Lambda_1)\subseteq C_{L_{\wo_{2mn}}(1,0)}(L_{\wo_{2m}}(n,0)).$$
If $m\equiv N\equiv 1  \ {\rm mod} \  2$, by (1) of Lemma \ref{add1}, the simple  vertex operator algebras $W$ such that $L_{\wo_{2m}}(n,0)\otimes L_{\wo_{n}}(2m,0)\subseteq W\subseteq L_{\wo_{2mn}}(1,0)$ are in one to one correspondence with the subgroups of $\Z/2\Z\times \Z/2\Z$. Then we can deduce that $$C_{L_{\wo_{2mn}}(1,0)}(L_{\wo_{2m}}(n,0))=L_{\wo_{n}}(2m,0)\oplus L_{\wo_{n}}(2m,2m\Lambda_1).$$
If $m\equiv 0 \ {\rm mod} \  2$ or $N\equiv 0 \ {\rm mod} \ 2$, set
$$
w^1=\prod_{i=1}^{m}\prod_{k=1}^{N}[\sqrt{-1}(\psi_{ik}(-\frac{1}{2}){\bf 1}-\sqrt{-1}\psi_{i,N+k}(-\frac{1}{2}){\bf 1})+(\psi_{m+i,k}(-\frac{1}{2}){\bf 1}-\sqrt{-1}\psi_{m+i,N+k}(-\frac{1}{2}){\bf 1})],$$
$$
\begin{array}{ll}
w^2=& \prod_{i=1}^{m}(\prod_{k=1}^{N-1}((\psi_{ik}(-\frac{1}{2}){\bf 1}+\sqrt{-1}\psi_{i,N+k}(-\frac{1}{2}){\bf 1})+\sqrt{-1}(\psi_{m+i,k}(-\frac{1}{2}){\bf 1}+\sqrt{-1}\psi_{m+i,N+k}(-\frac{1}{2}){\bf 1})))\\
& \cdot ((\psi_{iN}(-\frac{1}{2}){\bf 1}-\sqrt{-1}\psi_{i,2N}(-\frac{1}{2}){\bf 1})+\sqrt{-1}(\psi_{m+i,N}(-\frac{1}{2}){\bf 1}-\sqrt{-1}\psi_{m+i,2N}(-\frac{1}{2}){\bf 1})).
\end{array}
$$
Then it can be checked directly that
$$
\sum\limits_{j=1}^{n}\alpha_{kj}^{\vee}(r)w^i=\sum\limits_{j=1}^{n}e_{kj}(r)w^i=\sum\limits_{j=1}^{n}f_{kj}(r)w^i=0,
$$
$$
\beta_{s}^{\vee}(r)w^i=e_{s}(r)w^i=f_{s}(r)w^i=0,
$$
where $i=1,2$, $r\in\Z_{+}$, $1\leq k\leq m$, $1\leq s\leq N$.  Furthermore, we have
$$
\sum\limits_{j=1}^{n}\alpha_{lj}^{\vee}(0)w^1=0,  \ \sum\limits_{j=1}^{n}\alpha_{mj}^{\vee}(0)w^1
=-nw^1, \  \sum\limits_{j=1}^{n}f_{kj}(0)w^1=0, \ 1\leq k\leq m, \  \ 1\leq l\leq m-1,$$
$$
\beta_{s}^{\vee}(0)w^1=0, \ \beta_{N}^{\vee}(0)w^1=-2mw^1,  f_k(0)w^1=0, \ 1\leq s\leq N-1, \ 1\leq k\leq N,
$$
$$
\sum\limits_{j=1}^{n}\alpha_{lj}^{\vee}(0)w^2=0,  \ \sum\limits_{j=1}^{n}\alpha_{mj}^{\vee}(0)w^1
=nw^2, \  \sum\limits_{j=1}^{n}e_{kj}(0)w^1=0, \ 1\leq k\leq m, \  \ 1\leq l\leq m-1,$$
$$
\beta_{s}^{\vee}(0)w^2=0, \ \beta_{N-1}^{\vee}(0)w^1=-2mw^1,  e_k(0)w^1=0, \ s=1,\cdots,N-2, N,
\ 1\leq k\leq N.
$$
This means that the $L_{\wo_{2m}}(n,0)\otimes L_{\wo_{n}}(2m,0)$-submodule of $L_{\wo_{2mn}}(1,0)$ generated by $w^1$ is isomorphic to $L_{\wo_{2m}}(n, n\Lambda_{m})\otimes L_{\wo_{n}}(2m,2m\Lambda_N)$, and the $L_{\wo_{2m}}(n,0)\otimes L_{\wo_{n}}(2m,0)$-submodule of $L_{\wo_{2mn}}(1,0)$ generated by $w^2$ is isomorphic to $L_{\wo_{2m}}(n, n\Lambda_{m})\otimes L_{\wo_{n}}(2m,2m\Lambda_{N-1})$. By (2) of Lemma \ref{add1}, the simple vertex operator subalgebras $W$ such that $L_{\wo_{2m}}(n,0)\otimes L_{\wo_{n}}(2m,0)\subseteq W \subseteq L_{\wo_{2mn}}(1,0)$ are in one to one correspondence with the subgroups of $(\Z/2\Z)^3$ if $m\equiv N\equiv 0 \ {\rm mod} \ 2$, and  the subgroups of $\Z/2\Z\times \Z/4\Z$ if   $m\equiv 0 \ {\rm mod} \ 2$ and  $N\equiv 1 \ {\rm mod} \ 2$. By the fusion rules of affine vertex operator algebras of type $D$ (see \cite{Li3}),  we have
$$C_{L_{\wo_{2mn}}(1,0)}(L_{\wo_{2m}}(n,0))=L_{\wo_{n}}(2m,0)\oplus L_{\wo_{n}}(2m,2m\Lambda_1)$$
as desired.
\end{proof}

Recall that $G\leq Aut(L_{\wo_{2mn}}(1,0))$ is the abelian group generated by $\sigma_i$, $1\leq i\leq n$, where $\sigma_i$ is defined as in (\ref{aeq3}). The  following is the main result of this subsection.
\begin{theorem}  For $m\in\Z_{\geq 2}$ and $n=2N\in 2\Z_{+}$, we have
$$
C_{L_{\wo_{2m}}(1,0)^{\otimes n}}(L_{\wo_{2m}}(n,0))=(L_{\wo_{n}}(2m,0)\oplus L_{\wo_{n}}(2m,2m\Lambda_1))^G,$$
where  we mean $L_{\widehat{\frak{sl}_2}}(m,m)\otimes L_{\widehat{\frak{sl}_2}}(m,m)$ by $L_{\wo_{n}}(m,m\Lambda_1)$ if $n=4$.

\end{theorem}
\begin{proof} We may assume that $N\geq 3$.
By Lemma \ref {commutant2},  we have
$$
C_{L_{\wo_{2mn}}(1,0)}(L_{\wo_{2m}}(n,0))=C_{L_{\wo_{2mn}}(1,0)}(U)= L_{\wo_{n}}(2m,0)\oplus L_{\wo_{n}}(2m,2m\Lambda_1).$$
Recall that $L_{\wo_{n}}(2m,0)$ is generated by $\sum\limits_{k=1}^{2m}\psi_{ki}(-\frac{1}{2})\psi_{kj}(-\frac{1}{2}){\bf 1}$, $1\leq i,j\leq n$. Since for $1\leq r\leq n$,
$$
\sigma_r(\sum\limits_{k=1}^{2m}\psi_{ki}(-\frac{1}{2})\psi_{kj}(-\frac{1}{2}){\bf 1})=(-1)^{\delta_{ri}+\delta_{rj}}\sum\limits_{k=1}^{2m}\psi_{ki}(-\frac{1}{2})\psi_{kj}(-\frac{1}{2}){\bf 1},
$$
it follows that $L_{\wo_{n}}(2m,0)$ is invariant under $G$, that is, $G$ acts on $L_{\wo_{n}}(2m,0)$. By (\ref{aeq1}) and (\ref{aeq2}) we have
$$
\sigma_1u^{(1,N+1)}=v^{(1,N+1)}, \ \sigma_{N+1}u^{(1,N+1)}=v^{(1,N+1)}, \ \sigma_{j}u^{(1,N+1)}=u^{(1,N+1)}, \ j\neq 1, N+1.
$$
Notice that $L_{\wo_{n}}(2m,2m\Lambda_1)$ is generated by $u^{(1,N+1)}$ and $v^{(1,N+1)}$ is in the $\frak{so}_n$-submodule of $L_{\wo_{n}}(2m,2m\Lambda_1)$ generated by $u^{(1,N+1)}$, we know that $G$ acts on $L_{\wo_{n}}(2m,2m\Lambda_1)$, also.
Then the lemma  follows from Lemma \ref {auto1}.
\end{proof}

\subsubsection{The case $n=2N+1$, $N\in \Z_+$}\label{sec:4.1.4} In this subsection, we assume that $m\in\Z_{\geq 2}$, $n=2N+1\in 2\Z+1$ and  $N\in\Z_{+}$. Using Table 3 of \cite{KFPX} for $n\geq 5$ and noticing the fact that if $n=3$, ${\mathcal F}_3^{even}$ is isomorphic to $L_{\widehat{\frak{sl}_2}}(2,0)$ and
$$u^{(12)}=\prod_{k=1}^{2m}(\psi_{k1}(-\frac{1}{2})+\sqrt{-1}\psi_{k2}(-\frac{1}{2})){\bf 1}\in C_{L_{\wo_{6m}}(1,0)}(L_{\wo_{2m}}(3,0))$$
 generates the $L_{\widehat{\frak{sl}_2}}(4m,0)$-module $L_{\widehat{\frak{sl}_2}}(4m,4m)$,
  we can deduce the following result.
 \begin{lem}\label{commutant3} For $m\in\Z_{\geq 2}$,   we have
$$C_{L_{\wo_{6m}}(1,0)}(L_{\wo_{2m}}(3,0))=L_{\widehat{\frak{sl}_2}}(4m,0)+L_{\widehat{\frak{sl}_2}}(4m,4m),$$
 and
 $$C_{L_{\wo_{2mn}}(1,0)}(L_{\wo_{2m}}(n,0))=L_{\wo_{n}}(2m,0),$$
for $n=2N+1\geq 5$.

 \end{lem}
 By Lemma \ref{commutant3} and Lemma \ref{auto1}, we have
 \begin{theorem}
For $m\in\Z_{\geq 2}$,  we have
$$C_{L_{\wo_{2m}}(1,0)^{\otimes 3}}(L_{\wo_{2m}}(3,0))=(L_{\widehat{\frak{sl}_2}}(4m,0)\oplus L_{\widehat{\frak{sl}_2}}(4m,4m))^G,$$
and
 $$C_{L_{\wo_{2m}}(1,0)^{\otimes n}}(L_{\wo_{2m}}(n,0))=L_{\wo_{n}}(2m,0)^G,$$
for $n=2N+1\geq 5$.

\end{theorem}

\subsection{Type $B$}
Let $m\in\Z_{\geq 2},n\in\Z_{+}$.
We  now consider the Clifford algebra ${\mathcal Cl}_{(2m+1)n}$ generated by
$$
\psi_{ij}(r), \ r\in\Z+\frac{1}{2}, \ 1\leq i\leq 2m+1, \ 1\leq j\leq n,
$$
with  the  non-trivial relations
$$
[\psi_{ij}(r),\psi_{kl}(s)]_+=\delta_{ik}\delta_{jl}\delta_{r+s,0},
$$
where $1\leq i,k\leq 2m+1, 1\leq j,l\leq n$, $r,s\in\Z+\frac{1}{2}$.

Let ${\mathcal F}_{(2m+1)n}$ be the irreducible ${\mathcal Cl}_{(2m+1)n}$-module generated by the cyclic vector ${\bf 1}$ such that
$$
\psi_{ij}(r)=0, \ {\rm for} \ r>0, \ 1\leq i\leq 2m+1, \ 1\leq j\leq n.
$$
Define the following fields on ${\mathcal F}_{(2m+1)n}$
$$
\psi_{ij}(z)=\sum\limits_{r\in\Z}\psi_{ij}(r+\frac{1}{2})z^{-r-1}.
$$
The fields $\psi_{ij}(z)$, $1\leq i\leq 2m+1, 1\leq j\leq n$ generate on ${\mathcal F}_{(2m+1)n}$ the unique structure  of a simple vertex superalgebra \cite{FFR,K2,KW,Li1,AP}.
Recall from \cite{FF1} that
$$
{\mathcal F}_{(2m+1)n}^{even}\cong L_{\wo_{(2m+1)n}}(1,0).
$$
\subsubsection{The case $n=2$}  Let $\bigoplus_{i=1}^{2m+1}{\mathbb Z}\epsilon^{i}$  be a lattice such that $(\epsilon^i,\epsilon^j)=\delta_{ij}$. Set
$$
\gamma=\sum\limits_{i=1}^{2m+1}\epsilon^i.
$$
Let $V_{\Z 2\gamma}$ be the lattice vertex operator algebra associated with the lattice $L=\Z 2\gamma$. Let $\sigma$ be the automorphism of $V_{\Z 2\gamma}$ defined by
\begin{align*}
\sigma(\gamma(-n_{1})\gamma(-n_{2})\cdots \gamma(-n_{k})
e^{\alpha})=(-1)^{k}\gamma(-n_{1})\gamma(-n_{2})\cdots
\gamma(-n_{k}) e^{-\alpha}
\end{align*}
for $n_i\geq1, 1\leq i\leq k$ and $\alpha\in L$.
The following result follows from [AP].
\begin{theorem}
$$
C_{L_{\widehat{\frak{so}}_{2m+1}}(1,0)^{\otimes 2}}(L_{\widehat{\frak{so}}_{2m+1}}(2,0))\cong V_{\mathbb Z2\gamma}^{\sigma}.$$
\end{theorem}
\subsubsection{The case $n\geq 3$}
As in Subsection 4.1.2,
 we define the vertex operator algebra automorphism $\sigma_i$ of $L_{\wo_{(2m+1)n}}(1,0)$ for $1\leq i\leq n$ as follows:
$$
\sigma_{i}(\psi_{kj}(-\frac{1}{2})\psi_{rs}(-\frac{1}{2}){\bf 1})=(-1)^{\delta_{ij}+\delta_{is}}\psi_{kj}(-\frac{1}{2})\psi_{rs}(-\frac{1}{2}){\bf 1}.$$
Then $$\sigma_{i}^2={\rm id}.$$

Denote by $G$ the automorphism
group of $L_{\wo_{(2m+1)n}}(1,0)$ generated by $\{\sigma_{i}$, $1\leq i\leq n\}$.  Then $G$ is an abelian group. Set
$$
L_{\wo_{(2m+1)n}}(1,0)^{G}=\{v\in L_{\wo_{(2m+1)n}}(1,0) \mid g(v)=v, \ g\in G\}.$$

Similar to the proof of Lemma \ref{commutant2} and Lemma \ref{commutant3}, using Table 3 of \cite{KFPX} we can deduce that for $m\geq 2$ and $n\geq 4$,
$$
C_{L_{\widehat{\frak{so}}_{(2m+1)n}}(1,0)}(L_{\widehat{\frak{so}}_{2m+1}}(n,0))=L_{\widehat{\frak{so}}_{n}}(2m+1,0).
$$
Then as  discussions in Section \ref{sec:4.1.4}, we have
\begin{theorem}\label{main5}
For $m\geq 2$, we have
$$C_{L_{\widehat{\frak{so}}_{2m+1}}(1,0)^{\otimes 3}}(L_{\widehat{\frak{so}}_{2m+1}}(3,0))=L_{\widehat{\frak{sl}_2}}(4m+2,0)^G,$$
and
 $$
C_{L_{\widehat{\frak{so}}_{2m+1}}(1,0)^{\otimes n}}(L_{\widehat{\frak{so}}_{2m+1}}(n,0))=L_{\widehat{\frak{so}}_{n}}(2m+1,0)^G$$
for $n\geq 4$.

\end{theorem}
\begin{remark} By the above results,  \cite[Theorem 5.6]{Li4}, \cite[Theorem 1]{M}, and  \cite[Theorem 5.24]{CM}, the commutant vertex operator algebra $C_{L_{\widehat{\frak{so}_{m}}}(1,0)^{\otimes n}}(L_{\widehat{\frak{so}_{m}}}(n,0))$ is regular for $m\geq 4,n\geq 2$. Then the classification of irreducible modules of the commutant vertex operator algebra follows from \cite{DRX}.
\end{remark}

\end{document}